\newcommand{\CLASSINPUTtoptextmargin}{0.75in}
\newcommand{\CLASSINPUTbottomtextmargin}{1in}
\documentclass[conference]{IEEEtran}
\makeatletter
\def\ps@headings{%
\def\@oddhead{\mbox{}\scriptsize\rightmark \hfil \thepage}%
\def\@evenhead{\scriptsize\thepage \hfil \leftmark\mbox{}}%
\def\@oddfoot{}%
\def\@evenfoot{}}
\makeatother
\usepackage{amsfonts}%
\usepackage{amssymb}%
\usepackage{graphicx}%
\usepackage{amsmath}%
\newtheorem{theorem}{Theorem}

\newtheorem{lemma}{Lemma}

\newcommand{\comment}[1]{}

\ifCLASSINFOpdf
\else
\fi

\usepackage{framed}
\usepackage{algorithmic}

\begin{document}
\title{Online Learning in Opportunistic Spectrum Access: A Restless Bandit Approach}
\author{\IEEEauthorblockN{Cem Tekin, Mingyan Liu}
\IEEEauthorblockA{Department of Electrical Engineering and Computer Science\\
University of Michigan, Ann Arbor, Michigan, 48109-2122\\
Email: \{cmtkn, mingyan\}@umich.edu}
}

\maketitle

\begin{abstract}  
We consider an opportunistic spectrum access (OSA) problem where the time-varying condition of each channel (e.g., as a result of random fading or certain primary users' activities) is modeled as an arbitrary finite-state Markov chain.  At each instance of time, a (secondary) user probes a channel and collects a certain reward as a function of the state of the channel (e.g., good channel condition results in higher data rate for the user).  Each channel has potentially different state space and statistics, both unknown to the user, who tries to learn which one is the best as it goes and maximizes its usage of the best channel.  The objective is to construct a good online learning algorithm so as to minimize the difference between the user's performance in total rewards and that of using the best channel (on average) had it known which one is the best from a priori knowledge of the channel statistics (also known as the regret).  This is a classic exploration and exploitation problem and results abound when the reward processes are assumed to be {\em iid}.  Compared to prior work, the biggest difference is that in our case the reward process is assumed to be Markovian, of which {\em iid} is a special case.  In addition, the reward processes are {\em restless} in that the channel conditions will continue to evolve independent of the user's actions.  This leads to a restless bandit problem, for which there exists little result on either algorithms or performance bounds in this learning context to the best of our knowledge.  In this paper we introduce an algorithm that utilizes regenerative cycles of a Markov chain and computes a sample-mean based index policy, and show that under mild conditions on the state transition probabilities of the Markov chains this algorithm achieves logarithmic regret uniformly over time, and that this regret bound is also optimal.  We numerically examine the performance of this algorithm along with a few other learning algorithms in the case of an OSA problem with Gilbert-Elliot channel models, and discuss how this algorithm may be further improved (in terms of its constant) and how this result may lead to similar bounds for other algorithms. 

\end{abstract}
\IEEEpeerreviewmaketitle

\section{Introduction} \label{sec:intro}

In this paper we study the following opportunistic spectrum access (OSA) problem.  A (secondary) user has access to a set of $K$ channels, each of time-varying conditions as a result of random fading and/or certain primary users' activities.  Each channel is thus modeled as an arbitrary finite-state discrete-time Markov chain.  At each time step, the secondary user (simply referred to as {\em the user} for the rest of the paper for there is no ambiguity) probes a channel to find out its condition, and is allowed to use the channel in a way consistent with its condition.  For instance, good channel conditions result in higher data rates or lower power for the user and so on.  This is modeled as a reward collected by the user, the reward being a function of the state of the channel or the Markov chain.

Channels have potentially different state spaces and statistics, both unknown to the user.  The user will thus try to learn which one is the best and maximizes its usage of the best channel.  
Within this context, the player's performance is typically measured by the notion of {\em regret}. It is defined as the difference between the expected reward that can be gained by an ``infeasible'' or ideal policy, i.e., a policy that requires either a priori knowledge of some or all statistics of the arms or hindsight information, and the expected reward of the player's policy.  
The most commonly used infeasible policy is the {\em best single action} policy, that is optimal among all policies that continue to play the same arm.  An ideal policy could play for instance the arm that has the highest expected reward (which requires statistical information but not hindsight).  This type of regret is sometimes also referred to as the {\em weak regret}, see e.g., work by Auer et al. \cite{auer2}.  In this paper we will only focus on this definition of regret. 

The above can be cast as a single player multiarmed bandit problem, where the reward of each channel (also referred to as an {\em arm} in the bandit problem literature) is generated by a Markov chain with unknown statistics.  Furthermore, it is a {\em restless} bandit problem because the state of each Markov chain evolves independent of the action of the user (whether the channel is probed or not);  by contrast, in a classic multiarmed bandit problem the state of a Markov chain only evolves when it is acted upon and stays frozen otherwise (also referred to as {\em rested}).  The restless nature of the Markov chains follows naturally from the fact that channel conditions are governed by external factors like random fading, shadowing, and primary user activity. 

In the remainder of this paper a channel will also be referred to as an {\em arm}, the user as the {\em player}, and probing a channel as {\em playing or selecting an arm}.  This problem is a typical example of the tradeoff between {\em exploration} and {\em exploitation}.  On the one hand, the player needs to sufficiently explore all arms so as to discover with accuracy the best arm and avoid getting stuck playing an inferior one erroneously believed to be the best.  On the other hand, he needs to avoid spending too much time sampling the arms and collecting statistics and not playing the best arm often enough to get a high return.  

In most prior work on the class of multiarmed bandit problems, originally proposed by Robbins \cite{robbins2}, 
the rewards are assumed to be independently drawn from a fixed (but unknown) distribution.  
Its worth noting that with this iid assumption on the reward process, whether an arm is rested or restless is inconsequential for the following reasons.  
Since the rewards are independently drawn each time, whether an unselected arm remains still or continues to change does not affect the reward the arm produces the next time it is played whenever that may be.  
This is clearly not the case with Markovian rewards.  
In the rested case, since the state is frozen when an arm is not played, 
the state in which we next observe the arm is {\em independent} of how much time elapses before we play the arm again.   In the restless case, the state of an arm continues to evolve, 
thus the state in which we next observe it is now {\em dependent} on the amount of time that elapses between two plays of the same arm.  This makes the problem significantly more difficult.  

To the best of our knowledge, there has been no study of the restless bandits in this learning context, either in terms of algorithms or performance bounds.  Here lies the main contribution of the present study.  In this paper we give the first result on the existence of order-optimal policies for the above restless bandit problem.  Specifically, we introduce an algorithm that utilizes regenerative cycles of a Markov chain and computes a sample-mean based index policy, and show that under mild conditions on the state transition probabilities 
this algorithm achieves logarithmic regret uniformly over time. 



Below we briefly summarize the most relevant results in the literature.  
Lai and Robbins in  \cite{lai1} model rewards as single-parameter univariate densities 
and give a lower bound on the regret and construct policies that achieve this lower bound which are called {\em asymptotically efficient} policies. This result is extended by Anantharam et al. in \cite{anantharam2} to the case of playing more than one arm at a time. Using a similar approach Anantharam et al. in \cite{anantharam1} develops index policies that are asymptotically efficient for arms with rewards driven by finite, irreducible, aperiodic and rested Markov chains with identical state spaces and single-parameter families of stochastic transition matrices.  Agrawal in \cite{agrawal1} considers sample mean based index policies for the iid model that achieve $O(\log n)$ regret, where $n$ is the total number of plays. 
Auer et al. in \cite{auer} also proposes sample mean based index policies for iid rewards with bounded support; these are derived from \cite{agrawal1}, but are simpler than the those in \cite{agrawal1} and are not restricted to a specific family of distributions. These policies achieve logarithmic regret uniformly over time rather than asymptotically in time, but have bigger constant than that in \cite{lai1}.   
In \cite{tekin} it is shown that the index policy in \cite{auer} is order optimal for Markovian rewards drawn from {\em rested} arms but not restricted to single-parameter families, under some assumptions on the transition probabilities.

Other works such as \cite{liu2, anandkumar,krishnamachari} consider the iid reward case in a multiuser setting; players selecting the same arms experience collision according to a certain collision model.  
We would like to mention another class of multiarmed bandit problems in which the statistics about the problem are known a priori and the state is observed perfectly; these are thus optimization problems rather than learning problems. The rested case is considered by Gittins \cite{gittins1} and the optimal policy is proved to be an index policy which at each time plays the arm with highest Gittins' index, while Whittle introduced the restless bandit problem in \cite{whittle}.  The restless bandit problem does not have a known general solution though special cases may be solved.  For instance, a myopic policy is shown to be optimal when channels are identical and bursty in \cite{ahmad} for an OSA problem formulated as a restless bandit problem with each channel modeled as a two-state Markov chain (the Gilbert-Elliot model).

The remainder of this paper is organized as follows. In Section \ref{sec:problem} we formulate the single player restless bandit problem. In Section \ref{sec:policy} we introduce an algorithm based on regenerative cycles that employs sample-mean based indices.  The regret of this algorithm is analyzed and shown to be optimal in Section \ref{sec:policy2}. 
In Section \ref{sec:example} we numerically examine its performance along with a few other learning algorithms in the case of an OSA problem with Gilbert-Elliot channel models, and discuss how this algorithm may be further improved (in terms of its constant) and how this result may lead to similar bounds for other algorithms.  Finally, Section \ref{sec:conc} concludes the paper.


\section{Problem Formulation and Preliminaries} \label{sec:problem}

Consider $K$ arms (or channels) indexed by $i = 1,2, \cdots, K$.  The $i$th arm is modeled as a discrete-time, irreducible and aperiodic Markov chain with finite state space $S^i$. There is a stationary and positive reward associated with each state of each arm. 
Let $r^i_x$ denote the reward obtained from state $x$ of arm $i$, $x\in S^i$; this reward is in general different for different states. 
Let $P^i=\left\{p_{xy}^i, x,y \in S^i \right\}$ denote the transition probability matrix and $\boldsymbol{\pi}^i = \{\pi^i_x, x\in S^i \}$ the stationary distribution of arm $i$. 

Let $(P^i)'$ denote the {\em adjoint} of $P^i$ on $l_2(\pi)$, and let $\hat{P}^i=(P^i)'P$ denote the {\em multiplicative symmetrization} of $P^i$, where 
\begin{eqnarray*}
(p^i)'_{xy}=(\pi^i_y p^i_{yx})/\pi^i_x, ~ \forall x,y \in S^i.
\end{eqnarray*}

We will assume that the $P^i$s are such that $\hat{P}^i$s are irreducible.  To give a sense of how strong this assumption is, we note that one condition that guarantees that the $\hat{P}^i$s are irreducible is $p_{xx}>0, \forall x\in S^i, \forall i$.  For the application under consideration, this condition means that there is always positive probability for a channel to remain in the same state over one unit of time, which appears to be a natural and benign assumption\footnote{Alternatively we could adopt a stronger assumption that the Markov chains are aperiodic and reversible (note that aperiodicity and reversibility implies that the multiplicative symmetrization of $P^i$ is irreducible), in which case the same order results can be obtained with a different constant if we use a different large deviation bound from \cite{gillman} instead of Lemma 1.}. 


We assume the arms (or Markov chains) are mutually independent and are restless, i.e., their states will continue to evolve regardless of the user's actions.   The  mean reward of arm $i$, denoted by $\mu^i$, is the expected reward of arm $i$ under its stationary distribution: 
\begin{eqnarray}
\mu^i=\displaystyle\sum_{x \in S^i} r^i_x \pi^i_x ~. 
\end{eqnarray}
%
For convenience, we will use $^*$ in the superscript to denote the arm with the highest mean.  For instance, $\mu^*=\max_{1\leq i\leq K} \mu^i$, and so on.  We assume that the arm with the highest mean is unique. 

Consistent with the discrete-time Markov chain model, we will assume that the user's actions occur also in discrete time. 
For a policy $\alpha$ we define its regret $R^\alpha(n)$ as the difference between the expected total reward that can be obtained by playing the arm with the highest mean and the expected total reward obtained from using policy $\alpha$ up to time $n$.   Always playing the arm with the highest mean reward is referred to as the {\em best single-action policy}, and this arm will also be referred to as the {\em optimal} arm; accordingly the others will be referred to as {\em suboptimal} arms. 

Let $\alpha(t)$ denote the arm selected by policy $\alpha$ at $t$, $t=1, 2, \cdots$, and $x_{\alpha}(t)$ the state of arm $\alpha(t)$ at time $t$.  Then we have
\begin{eqnarray}
R^\alpha(n)=n\mu^*-E^\alpha\left[\displaystyle\sum_{t=1}^n r^{\alpha(t)} _{x_{\alpha}(t)} \right] ~. 
\end{eqnarray}
The objective is to examine how the regret $R^\alpha(n)$ behaves as a function of $n$ for a given policy $\alpha$ and to construct a policy whose regret is order-optimal, through appropriate bounding.  As we will show and as is commonly done, the key to bounding $R^\alpha(n)$ is to bound the expected number of plays of any suboptimal arm.
%

Our analysis utilizes the following known results on Markov chains; the proofs are not reproduced here for brevity.  The first is a result by Lezaud \cite{lezaud} that bounds the probability of a large deviation from the stationary distribution.
\begin{lemma}\label{lemma3} 
[Theorem 3.3 from \cite{lezaud}] Consider a finite-state, irreducible Markov chain $\left\{X_t\right\}_{t \geq 1}$ with state space $S$, matrix of transition probabilities $P$, an initial distribution $\mathbf{q}$ and stationary distribution $\mathbf{\pi}$. Let $N_{\mathbf{q}}=\left\|(\frac{q_x}{\pi_x}, x\in S)\right\|_2$. Let $\hat{P}=P'P$ be the multiplicative symmetrization of $P$ where $P'$ is the adjoint of $P$ on $l_2(\pi)$. Let $\epsilon=1-\lambda_2$, where $\lambda_2$ is the second largest eigenvalue of the matrix $\hat{P}$. $\epsilon$ will be referred to as the eigenvalue gap of $\hat{P}$. Let $f:S\rightarrow \mathcal{R}$ be such that $\sum_{y \in S} \pi_y f(y) =0$, $\left\|f\right\|_{\infty} \leq 1$ and $0<\left\|f\right\|^2_2 \leq 1$. If $\hat{P}$ is irreducible, then for any positive integer $n$ and all $0 < \gamma \leq 1$
\begin{eqnarray*}
P\left(\frac{\sum_{t=1}^n f(X_t)}{n} \geq \gamma \right) \leq N_q \exp\left[-\frac{n \gamma^2 \epsilon}{28}\right] ~. 
\end{eqnarray*}
\end{lemma}

The second is a result by Bremaud, which can be found in \cite{bremaud}.
\begin{lemma}
If $\left\{X_n\right\}_{n \geq 0}$ is a positive recurrent homogeneous Markov chain with state space $S$, stationary distribution $\pi$ and $\tau$ is a stopping time that is finite almost surely for which $X_{\tau}=x$ then for all $y \in S$
\begin{eqnarray*}
E\left[ \sum_{t=0}^{\tau-1} I(X_t=y) | X_0=x \right] = E[\tau | X_0=x]\pi_y ~. 
\end{eqnarray*}
\label{lastlemma}
\end{lemma}

In the next two sections we first present a policy, referred to as the {\em regenerative cycle algorithm}, and then analyze its regret. 
\section{Regenerative Cycle Algorithm (RCA)} \label{sec:policy}

In this section we present an algorithm called the {\em regenerative cycle algorithm} (RCA), and prove in the next section that this algorithm guarantees a logarithmic growth of the regret uniformly over time 
under mild assumptions on the state transition probabilities and the rewards. 

As the name suggests, this algorithm operates on regenerative cycles.  
In essence what the algorithm does is to construct sample paths of each arm solely using those observed within regenerative cycles while discarding the rest in its estimation of the quality of an arm (in the form of an {\em index}).  The reason behind such a construction has to do with the restless nature of the arms.  As noted in the introduction, since each arm continues to evolve according to the Markov chain regardless of the user's action, the probability distribution of the reward we get by playing an arm is a function of the amount of time that has elapsed since the last time we played the same arm.  Since we play one arm at a time, the arms become coupled (in terms of the probability distributions of the rewards).  While this certainly does not affect our ability to collect rewards, it makes it extremely hard to analyze the estimated quality (or the index) of an arm calculated based on rewards collected this way.  

However, if instead of the actual sample path of all observations from an arm, we limit ourselves to a sample path constructed (or rather stitched together) using only the observations from regenerative cycles, then this sample path essentially has the same statistics as the original Markov chain due to the renewal property and one can now use the sample mean of the rewards from the regenerative sample paths to approximate the mean reward under stationary distribution.  



Figure \ref{fig:RCA} illustrates one possible realization of this algorithm.  As shown, RCA operates in blocks.  Within a block, the algorithm plays the same arm in each time slot (arm $i$ in the first block in this example) till a certain pre-specified state (say $\gamma^i$) is observed.  Upon this observation we enter a regenerative cycle and continue to play till the same state $\gamma^i$ is observed a second time.  This marks the end of the block labeled ``play arm $i$''.  At the end of each block, the algorithm computes an index for all arms and selects the one with the highest index to play in the next block (arm $k$ shown in the figure).  It follows that the block length is a random variable.   

For the purpose of index computation and subsequent analysis, each block is further broken into three sub-blocks (SBs).  SB1 consists of all time slots from the beginning of the block to right before the first visit to $\gamma^i$; SB2 includes all time slots from the first visit to $\gamma^i$ up to but excluding the second visit to state $\gamma^i$; SB3 consists of a single time slot with the second visit to $\gamma^i$.  These are also shown in Figure \ref{fig:RCA}.  The key to the algorithm is for each arm to single out only observations within SB2's in each block and virtually assemble them (these are highlighted with thick lines).  Because of the regenerative nature of the Markov chain, once put together, the resulting sample path has exactly the same statistics as given by the transition probability matrix $P^i$; this results in a tractable problem. 

\begin{figure}[t]
\vspace{0.08in}
\includegraphics[width=3.5in]{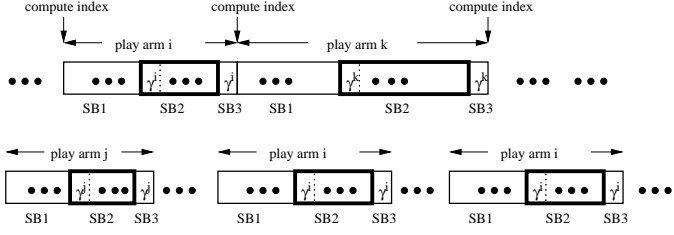}
\caption{Example realization of RCA}
\vspace{-15pt}
\label{fig:RCA}
\end{figure}

Throughout our discussion, we will consider a horizon of $n$ time slots.  A list of notations used is summarized as follows; some are also marked on Figure \ref{fig:RCA-notation} for convenience: 
\begin{figure}[h]
\includegraphics[width=3.5in]{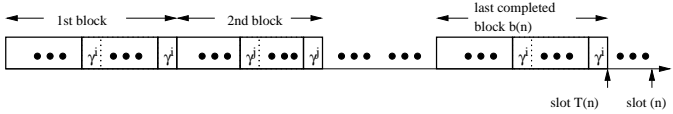}
\vspace{-15pt}
\caption{Running RCA over a period of $n$ slots}
\vspace{-10pt}
\label{fig:RCA-notation}
\end{figure}
\begin{itemize}
\item $\gamma^i$: state that determine the regenerative cycles for arm $i$.
\item $\alpha(b)$: the arm played in block $b$.
\item $b(n)$: total number of completed blocks up to time $n$.
\item $T(n)$: time at the end of the last completed block.
\item $T^i(n)$: total number of times (slots) arm $i$ is played up to time $T(n)$.
\item $B^i(b)$: total number of blocks within the first $b$ blocks in which arm $i$ is played.

\item $X^i_1(j)$: vector of observed states from SB1 of the $j$th block in which arm $i$ is played; 
it is empty if the first observed state is $\gamma^i$.
\item $X^i_2(j)$: vector of observed states from SB2 of the $j$th block in which arm $i$ is played; 
\item $X^i(j)$: vector of observed states from the $j$th block in which arm $i$ is played. Thus we have
$X^i(j)=[X^i_1(j), X^i_2(j), \gamma^i] $. 

\item $t(b)$: time at the end of block $b$; 
$t(b) = \sum_{i=1}^K \sum_{j=1}^{B^i(b)} |X^i(j)| $. 

\item $T^i(t(b))$: total number of time slots arm $i$ is played up to time $t(b)$. Thus 
$T^i(t(b)) =  \sum_{j=1}^{B^i(b)} |X^i(j)|$. 
Also note that $T^i(t(b(n))) = T^i(n)$. 

\item $t_2(b)$: total number of time slots spent in SB2 up to block $b$. Thus 
$t_2(b) = \sum_{i=1}^K \sum_{j=1}^{B^i(b)} |X^i_2(j)| $. 

\item $r^i(k)$: the reward from arm $i$ when it's played for the $k$-th time, counting only those times played during an SB2. 

\item $T^i_2(t_2(b))$: total number of time slots arm $i$ is played during SB2 up to block $b$. Thus 
$T^i_2(t_2(b)) = \sum_{j=1}^{B^i(b)} |X^i_2(j)| $. 
\end{itemize}

RCA computes and updates the value of an {\em index} $g^i$ for each arm $i$ at the end of block $b$, based on the total reward obtained from arm $i$ during all SB2 as follows: 
\begin{figure}[ht]
\vspace{0.08in}
\fbox {
\begin{minipage}{\columnwidth }
{Regenerative Cycle Algorithm (RCA):}\hspace{-30pt}
\begin{algorithmic}[1]
\STATE {Initialize: $b=1, t=0, t _2= 0, T^i_2=0, r^i=0, \forall i=1,\cdots,K$}

\FOR {$b\leq K$}
\STATE{play arm $b$; set $\gamma^b$ to be the first state observed} 
\STATE{$t := t+1$; $t_2 := t_2 + 1$; $T_2^b := T_2^b + 1$; $r^b := r^b + r^b_{\gamma^i}$}
\STATE{play arm $b$; denote observed state as $x$}
\WHILE{$x\neq \gamma^b$}
\STATE{$t := t+1$; $t_2 := t_2 + 1$; $T_2^b := T_2^b + 1$; $r^b := r^b + r^b_{x}$}
\STATE{play arm $b$; denote observed state as $x$}
\ENDWHILE
\STATE{$b:=b+1$; $t:=t+1$}
\ENDFOR 

\FOR{$j=1$ to $K$}
\STATE{compute index $g^j := \frac{r^j}{T^j_2} + \sqrt{\frac{L\ln{t_2}}{T^j_2}}$} 
\STATE{$j++$} 
\ENDFOR

%
\STATE{$i:=\arg\max_{j} g^j$}

\WHILE{(1)}
\STATE{play arm $i$; denote observed state as $x$}
\WHILE{$x\neq \gamma^i$}
\STATE{$t:=t+1$} 
\STATE{play arm $i$; denote observed state as $x$}
\ENDWHILE

\STATE{$t := t+1$; $t_2 := t_2 + 1$; $T_2^i := T_2^i + 1$; $r^i := r^i + r^i_{x}$} 
\STATE{play arm $i$; denote observed state as $x$}
\WHILE{$x\neq \gamma^i$}
\STATE{$t := t+1$; $t_2 := t_2 + 1$; $T_2^i := T_2^i + 1$; $r^i:= r^i + r^i_{x}$} 
\STATE{play arm $i$; denote observed state as $x$}
\ENDWHILE
\STATE{$b:=b+1$; $t:=t+1$}
\FOR{$j=1$ to $K$}
\STATE{compute index $g^j := \frac{r^j}{T^j_2} + \sqrt{\frac{L\ln{t_2}}{T^j_2}}$} 
\STATE{$j++$} 
\ENDFOR
\STATE{$i:=\arg\max_{j} g^j$}

\ENDWHILE

\end{algorithmic}
\end{minipage}
} \caption{Pseudocode of RCA} \label{fig:RCA_algo}
\vspace{-10pt}
\end{figure}
\begin{eqnarray}
g^i_{t_2(b), T^i_2(t_2(b))}= \bar{r}^i(T^i_2(t_2(b))) + \sqrt{\frac{L \ln t_2(b)}{T^i_2(t_2(b))}}, \label{eqn1redifine}
\end{eqnarray} 
where $L$ is a constant, and 
\begin{eqnarray}
\bar{r}^i(T^i_2(t_2(b))=\frac{r^i(1)+r^i(2)+...+r^i(T^i_2(t_2(b)))}{T^i_2(t_2(b))} \nonumber
\end{eqnarray}
denotes the sample mean of the reward collected during an SB2: $X^i_2(1), X^i_2(2), \cdots, X^i_2(B^i(b))$ (this is arm $i$'s total reward over the total number of times it's played).   The second term in the index computation serves the purpose of {\em exploration}: the relative uncertainty of the mean reward of an arm grows as the arm is not played. This index definition is similar to that proposed in \cite{auer}, but computed only over SB2s. 
RCA is formally given in Figure \ref{fig:RCA_algo}.  In this description the algorithm continues indefinitely, but can obviously be stopped at anytime that some desired horizon is reached. 

\comment{
\hspace{0.2in}\begin{framed}
\noindent\textbf{Regenerative Cycle Algorithm (RCA)} \\

\noindent\textbf{Initialization:} $b=1$, $t_2=0$, $T^i_2(t_2)=0, \forall i=1,\cdots,K$\\

for ($b \leq K$) \\
\textbf{Initialization:} Play arm $b$. Set $\gamma^b$ to the first state observed. \\
	\indent $t_2$=$t_2+1$, $T^b_2(t_2)=T^b_2(t_2)+1$.\\
\textbf{Loop:}\\
	\textbf{if} observed state is not $\gamma^b$ keep playing the same arm.\\
	\indent  $t_2$=$t_2+1$, $T^b_2(t_2)=T^b_2(t_2)+1$\\
	\textbf{else if} observed state is $\gamma^b$  \\
	\indent  $b=b+1$\\

while ($b>K$) \\ 
In block $b$ play arm with the highest index according to the indices calculated at the end of block $b-1$.\\
\textbf{Loop:}\\
	\textbf{if} $\gamma^{\alpha(b)}$ is not observed up to current time in block $b$\\
	\indent wait\\
	\textbf{else if} $\gamma^{\alpha(b)}$ is observed once in block $b$ up to current time\\
	\indent $t_2$=$t_2+1$, $T^{\alpha(b)}_2(t_2)=T^{\alpha(b)}_2(t_2)+1$\\
	\textbf{else if} $\gamma^{\alpha(b)}$ is observed for the second time in block $b$\\
	\indent update the indices,$b=b+1$\\
\end{framed}
} 

Its worth noting that RCA also collects reward during SB1 and SB3.  However, the computation of the indices only relies on SB2.  The reason becomes clearer in the next section where we analyze its regret and show that it grows at most logarithmically in $n$. 


\section{Regret analysis of RCA} \label{sec:policy2}

\comment{
First, we state the following result by Lezaud \cite{lezaud}, which bounds the probability of a large deviation from the stationary distribution.
\begin{lemma}\label{lemma3} 
[Theorem 3.3 from \cite{lezaud}] Consider a finite-state, irreducible Markov chain $\left\{X_t\right\}_{t \geq 1}$ with state space $S$, matrix of transition probabilities $P$, an initial distribution $\mathbf{q}$ and stationary distribution $\mathbf{\pi}$. Let $N_{\mathbf{q}}=\left\|(\frac{q_x}{\pi_x}, x\in S)\right\|_2$. Let $\hat{P}=P'P$ be the multiplicative symmetrization of $P$ where $P'$ is the adjoint of $P$ on $l_2(\pi)$. Let $\epsilon=1-\lambda_2$, where $\lambda_2$ is the second largest eigenvalue of the matrix $\hat{P}$. $\epsilon$ will be referred to as the eigenvalue gap of $\hat{P}$. Let $f:S\rightarrow \mathcal{R}$ be such that $\sum_{x \in S} \pi_x f(x) =0$, $\left\|f\right\|_{\infty} \leq 1$ and $0<\left\|f\right\|^2_2 \leq 1$. If $\hat{P}$ is irreducible, then for any positve integer $n$ and all $0 < \gamma \leq 1$
\begin{eqnarray*}
P\left(\frac{\sum_{t=1}^n f(X_t)}{n} \geq \gamma \right) \leq N_q \exp\left[-\frac{n \gamma^2 \epsilon}{28}\right]
\end{eqnarray*}
\begin{proof} 
See Theorem 3.3 of \cite{lezaud}.
\end{proof}
\end{lemma}
} 

We begin by bounding the expected number of plays of a suboptimal arm. 

\begin{theorem}
\label{theorem1}
Assume all arms are finite-state, irreducible, aperiodic Markov chains whose transition probability matrices have irreducible multiplicative symmetrizations \comment{\footnote{Note that instead of the assumption that the multiplicative symmetrizations are irreducible, we can assume that the transition probability matrices are reversible provided that we use a large deviation bound from \cite{gillman} rather than from \cite{lezaud} as we have done in the present paper.}}and assume all rewards are positive. 
Let $\pi^i_{\min} = \min_{x \in S^i} \pi^i_x$, $\pi_{\min} = \min_{1 \leq i \leq K} \pi^i_{\min}$, $r_{\max}=\max_{x \in S^i, 1 \leq i \leq K} r^i_x$, $S_{\max}=\max_{1 \leq i \leq K} |S^i|$, $\hat{\pi}_{\max} = \max_{x \in S^i, 1 \leq i \leq K} \left\{\pi^i_x, 1-\pi^i_x\right\}$, $\epsilon_{\min}= \min_{1\leq i \leq K} \epsilon^i$, $M^i_{\max}=\max_{x,y \in S^i, x \neq y} M^i_{x,y}$, where $\epsilon^i$ is the eigenvalue gap of the multiplicative symmetrization of the transition probability matrix of the $i$th arm and $M^i_{x,y}$ is the mean hitting time of state $y$ starting from an initial state $x$ for the $i$th arm. Then for a player using RCA with a constant $L \geq 112 S^2_{\max} r^2_{\max} \hat{\pi}^2_{\max} /\epsilon_{\min}$ in (\ref{eqn1redifine}), we have 
\begin{eqnarray}
&& \sum_{i:\mu^i < \mu^*} (\mu^*-\mu^i) E[T^i(n)] \nonumber \\
&\leq&  4L\displaystyle\sum_{i: \mu^i < \mu^*} \frac{D_i \ln n}{(\mu^*-\mu^i)} +  \displaystyle\sum_{i: \mu^i < \mu^*} (\mu^*-\mu^i) D_i C_i ~, \nonumber 
\end{eqnarray}
where 
\begin{eqnarray*}
C_i &=& \left(1+\frac{(|S^i|+|S^*|)\beta} {\pi_{\min}}\right), ~~ \beta = \sum_{t=1}^{\infty} t^{-2}\\
D_i &=& \left(\frac{1}{\pi^i_{\min}} + M^i_{\max} +1\right).
\end{eqnarray*} 
\end{theorem}

\begin{proof}
Throughout the proof all quantities pertain to RCA, which will be denoted by $\alpha$ and suppressed from the superscript whenever there is no ambiguity. 
%
%
Let $c_{t,s}=\sqrt{L\ln t/s}$, and let $l$ be any positive integer. Then,
\begin{eqnarray}
&& B^i(b) = 1+\displaystyle\sum_{m=K+1}^{b} I(\alpha(m)=i) \nonumber \\
& \leq&  l+\displaystyle\sum_{m=K+1}^{b} I(\alpha(m)=i, B^i(m-1) \geq l)  \nonumber \\
& \leq&  l+\displaystyle\sum_{m=K+1}^{b} I(g^*_{t_2(m-1), T^*_2(t_2(m-1))} \nonumber \\
&& ~~~~~~~~~~~~~~~\leq g^i_{t_2(m-1), T^i_2(t_2(m-1))}, B^i(m-1) \geq l) \nonumber \\
&\leq&  l+\sum_{m=K+1}^{b} I\left(\min_{1 \leq s \leq t_2(m-1)} g^*_{t_2(m-1), s} \right. \nonumber \\
&& ~~~~~~~~~~~~~~~~\leq \left. \max_{t_2(l) \leq s_i \leq t_2(m-1)} g^i_{t_2(m-1), s_i} \right) \nonumber
\end{eqnarray}
\\
\begin{eqnarray}
&\leq&  l+ \displaystyle\sum_{m=K+1}^{b}\hspace{-5pt}\sum_{s=1}^{t_2(m-1)}\sum_{s_i=t_2(l)}^{t_2(m-1)} I(g^*_{t_2(m),s}\leq g^i_{t_2(m), s_i})  \label{jul21} \\
&\leq& l + \sum_{t=1}^{t_2(b)} \sum_{s=1}^{t-1} \sum_{s_i=l}^{t-1} I(g^*_{t, s} \leq g^i_{t, s_i}) \label{eqn1} 
\end{eqnarray} 
where as given in (\ref{eqn1redifine}), 
$g^i_{t,s}=\bar{r}^i(s)+c_{t,s}$.
The inequality in (\ref{eqn1}) follows from the fact that the outer sum in (\ref{eqn1}) is over time while the outer sum in (\ref{jul21}) is over blocks and each block lasts at least two time slots.

We now show that $g^*_{t, s} \leq g^i_{t, s_i}$ implies that at least one of the following holds: 
\begin{eqnarray}
\bar{r}^*(s) &\leq& \mu^* - c_{t,s} \label{eqn2}\\ 
\bar{r}^i(s_i) &\geq& \mu^i+c_{t,s_i} \label{eqn3}\\
\mu^* &<& \mu^i+2c_{t,s_i} \label{eqn4}.
\end{eqnarray}
This is because if none of the above holds, then we must have 
\begin{eqnarray*}
g^*_{t, s} = \bar{r}^*(s)+c_{t,s} > \mu^* \geq \mu^i+2c_{t,si} > \bar{r}^i(s_i)+c_{t,s_i} = g^i_{t, s_i}, 
\end{eqnarray*}
which contradicts $g^*_{t, s} \leq g^i_{t, s_i}$. 

If we choose $s_i \geq 4L\ln(t_2(b))/(\mu^*-\mu^i)^2$, then
%
$2c_{t,s_i}\leq \mu^*- \mu^i$ for $t \leq t_2(b)$, 
which means (\ref{eqn4}) is false, and therefore at least one of (\ref{eqn2}) and (\ref{eqn3}) is true with this choice of $s_i$.  
%
We next take $l=\left\lceil \frac{4L\ln t_2(b)}{(\mu^*-\mu^i)^2}\right\rceil$, and proceed from (\ref{eqn1}).  Taking expectation on both sides and relaxing the outer sum in (\ref{eqn1}) from $t_2(b)$ to $\infty$,
\begin{eqnarray*}
E[B^i(b)] &\leq& \left\lceil \frac{4L\ln t_2(b)}{(\mu^*-\mu^i)^2}\right\rceil \\
&+& \displaystyle\sum_{t=1}^{\infty}\sum_{s=1}^{t-1}\sum_{s_i=\left\lceil \frac{4L\ln t_2(b)}{(\mu^*-\mu^i)^2}\right\rceil}^{t-1} P(\bar{r}^*(s)\leq\mu^*-c_{t,s}) \\
&+& \displaystyle\sum_{t=1}^{\infty}\sum_{s=1}^{t-1}\sum_{s_i=\left\lceil \frac{4L\ln t_2(b)}{(\mu^*-\mu^i)^2}\right\rceil}^{t-1} P(\bar{r}^i(s_i) \geq \mu^i+c_{t,s_i}). 
\end{eqnarray*}

Consider an initial distribution ${\bf q}^i$ for the $i$th arm.  We have: 
\begin{eqnarray*}
&& N_{\mathbf{q}^i}=\left\|\left(\frac{q_{y}^i}{\pi_{y}^i},y\in S^i\right)\right\|_2 \leq \sum_{y \in S^i} \left\|\frac{q_{y}^i}{\pi_{y}^i}\right\|_2 \leq \frac{1}{\pi_{\min}} , 
\end{eqnarray*}
where the first inequality follows from Minkowski inequality.
Let $n^i_y(t)$ denote the number of times state $y$ of arm $i$ is observed during all SB2s up to the $t$th play. Then,
\begin{eqnarray} 
&& P(\bar{r}^i(s_i) \geq \mu^i+c_{t,s_i}) \nonumber \\
&=& P\left( \sum _{y \in S^i} r^i_y n^i_y(s_i) \geq s_i  \sum _{y \in S^i} r^i_y \pi^i_y+s_i c_{t,s_i} \right) \nonumber \\
&=& P\left( \sum_{y \in S^i} (r^i_y n^i_y(s_i) -r^i_y s_i \pi^i_y ) \geq s_i c_{t,s_i} \right) \nonumber 
\end{eqnarray}
\begin{eqnarray}
\vspace{0.02in}
&=& P( \sum_{y \in S^i} (-r^i_y n^i_y(s_i) + r^i_y s_i \pi^i_y ) \leq - s_i c_{t,s_i} ) ~. \label{correction1}
\end{eqnarray}
Consider a sample path $\omega$ and the events
\begin{eqnarray*}
A &=&\left\{\omega : \sum_{y \in S^i} (-r^i_y n^i_y(s_i)(\omega) + r^i_y s_i \pi^i_y ) \leq - s_i c_{t,s_i} \right\} ~,\\
B&=& \bigcup_{y \in S^i} \left\{\omega : -r^i_y n^i_y(s_i)(\omega) + r^i_y s_i \pi^i_y \leq - \frac{s_i c_{t,s_i}}{|S^i|}\right\} ~. 
\end{eqnarray*}
If $\omega \notin B$ then,
\begin{eqnarray*}
&& -r^i_y n^i_y(s_i)(\omega) + r^i_y s_i \pi^i_y > - \frac{s_i c_{t,s_i}}{|S^i|}, \ \forall y \in S^i \\
\Rightarrow && \sum_{y \in S^i} (-r^i_y n^i_y(s_i)(\omega) + r^i_y s_i \pi^i_y ) > - s_i c_{t,s_i}
\end{eqnarray*}
Thus $\omega \notin A$, and $P(A) \leq P(B)$. Then continuing from (\ref{correction1}):  
\begin{eqnarray}
&& P(\bar{r}^i(s_i) \geq \mu^i+c_{t,s_i}) \nonumber \\
&\leq& \sum_{y \in S^i} P\left( -r^i_y n^i_y(s_i) + r^i_y s_i \pi^i_y  \leq - \frac{s_ic_{t,s_i}}{|S^i|} \right) \nonumber \\
&=& \sum_{y \in S^i} P\left( r^i_y n^i_y(s_i) - r^i_y s_i \pi^i_y  \geq \frac{s_ic_{t,s_i}}{|S^i|} \right) \nonumber \\
&\leq& \sum_{y \in S^i} N_{q^i} t^{-\frac{L \epsilon^i}{28 (|S^i| r^i_y \hat{\pi}^i_y)^2}}  \label{n1eqn} \\
&\leq& \frac{|S^i|}{\pi_{\min}} t^{-\frac{L \epsilon_{\min}}{28 S^2_{\max} r^2_{\max} \hat{\pi}^2_{\max}}} \label{eqn8},
\end{eqnarray}
where (\ref{n1eqn}) follows from letting
\begin{eqnarray*}
\gamma &=& \frac{c_{t,s_i}}{|S^i| r^i_y \hat{\pi}^i_y},  ~~ \hat{\pi}^i_y=\max \left\{\pi^i_y, 1-\pi^i_y\right\}   \\
f(X^i_t) &=& \frac{I(X^i_t =y)-\pi^i_y}{\hat{\pi}^i_y},
\end{eqnarray*}
and using Lemma \ref{lemma3} (note $\hat{P}^i$ is irreducible), which gives 
\begin{eqnarray}
&& P\left( n^i_y(s_i) -s_i \pi^i_y  \geq \frac{s_ic_{t,s_i}}{|S^i| r^i_y} \right) \nonumber \\
&=& P\left(\frac{\sum_{t=1}^{s_i} I(X^i_t=y) -s_i \pi^i_y}{\hat{\pi}^i_y s_i} \geq \frac{c_{t,s_i}}{|S^i| r^i_y \hat{\pi}^i_y}\right) \nonumber \\ 
&\leq& N_{q^i} t^{-\frac{L \epsilon^i}{28 (|S^i| r^i_y \hat{\pi}^i_y)^2}} 
\end{eqnarray}
We note that for $\gamma >1$ the deviation probabiltiy is zero so the bound still holds.

Similarly, we have 
\begin{eqnarray} 
&& P(\bar{r}^*(s)\leq\mu^*-c_{t,s}) \nonumber \\
&= & P(\displaystyle \sum_{y \in S^*} r^*_y (n^*_y(s)-s \pi^*_y) \leq -sc_{t,s}) \nonumber \\
&\leq& \sum_{y \in S^*} P(r^*_y n^*_y(s)-r^*_y s \pi^*_y \leq -\frac{sc_{t,s}}{|S^*|} ) \nonumber \\
&=& \hspace{-4pt}\sum_{y \in S^*}\hspace{-4pt} P(r^*_y(s-\displaystyle\sum_{x \neq y} n^*_x(s)) - r^*_ys(1- \displaystyle\sum_{x \neq y} \pi^*_x) \leq -\frac{sc_{t,s}}{|S^*|}) \nonumber 
\end{eqnarray}
\begin{eqnarray}
\vspace{0.04in}
&=& \sum_{y \in S^*} P(r^*_y \displaystyle\sum_{x \neq y} n^*_x(s)- r^*_y s\displaystyle\sum_{x \neq y} \pi^*_x \geq \frac{sc_{t,s}}{|S^*|})  \nonumber \\
&\leq& \sum_{y \in S^*} N_{q^*} t^{-\frac{L \epsilon^*}{28 (|S^*|r^*_y \hat{\pi}^*_y)^2}} \label{eqn6} \\
&\leq& \frac{|S^*|}{\pi_{\min}} t^{-\frac{L \epsilon_{\min}}{28 S^2_{\max} r^2_{\max} \hat{\pi}^2_{\max}}} \label{eqn7}  
\end{eqnarray}
where (\ref{eqn6}) again follows from Lemma \ref{lemma3}.
Since
\begin{eqnarray}
&& \frac{|S^i|+|S^*|}{\pi_{\min}} \sum_{t=1}^{\infty}\sum_{s=1}^{t-1}\sum_{s_i=1}^{t-1} t^{-\frac{L \epsilon_{\min}}{28 S^2_{\max} r^2_{\max} \hat{\pi}^2_{\max}}} \nonumber \\
&=&\frac{|S^i|+|S^*|}{\pi_{\min}} \sum_{t=1}^{\infty} t^{-\frac{L \epsilon_{\min} - 56 S^2_{\max} r^2_{\max} \hat{\pi}^2_{\max}}{28 S^2_{\max} r^2_{\max} \hat{\pi}^2_{\max}}} \nonumber \\
&\leq& \frac{|S^i|+|S^*|}{\pi_{\min}} \sum_{t=1}^{\infty} t^{-2}, \label{newlabel1}
\end{eqnarray}
from (\ref{eqn8}) and (\ref{eqn7}), given $b(n)=b$ we have  
\begin{eqnarray}
E[B^i(b(n))|b(n)=b] \leq \left\lceil \frac{4L\ln t_2(b)}{(\mu^*-\mu^i)^2}\right\rceil + \frac{(|S^i|+|S^*|)\beta}{\pi_{\min}},\nonumber
\end{eqnarray}
for all suboptimal arms.  The inequality in (\ref{newlabel1}) follows from the assumption $L \geq 112 S^2_{\max} r^2_{\max} \hat{\pi}^2_{\max} /\epsilon_{\min}$.
Therefore,
\begin{eqnarray}
E[B^i(b(n))] \leq \frac{4L\ln n}{(\mu^*-\mu^i)^2} +1+ \frac{(|S^i|+|S^*|)\beta}{\pi_{\min}}, \label{eqn9}
\end{eqnarray}
since $n \geq t_2(b(n))$ almost surely.

Note that all the quantities in computing the indices and the probabilities in above comes from the intervals $X^i_2(1), X^i_2(2), \cdots \forall i \in \left\{1,\cdots,K\right\}$. Since these intervals begin with state $\gamma^i$ and end with a return to $\gamma^i$, by the strong Markov property the process at these stopping times have the same distribution as the original process. Moreover by connecting these intervals together we form a continuous sample path which can be viewed as a sample path generated by a Markov chain with an transition matrix identical to the original arm. This is the reason why we can apply Lezaud's bound to this Markov chain.

The total number of plays of arm $i$ at the end of block $b(n)$ is equal to the total number of plays of arm $i$ during the regenerative cycles of visiting state $\gamma^i$ plus the total number of plays before entering the regenerative cycles plus one more play resulting from the last play of the block which is state $\gamma^i$.  This gives: 
\begin{eqnarray}
E[T^i(n)] \leq \left(\frac{1}{\pi^i_{\min}} + M^i_{\max} +1\right)E[B^i(b(n))] ~. \nonumber
\end{eqnarray}
Thus,
\begin{eqnarray}
&&\displaystyle\sum_{i:\mu^i < \mu^*} (\mu^*-\mu^i) E[T^i(n)] \nonumber \\
&\leq& 4L\displaystyle\sum_{i: \mu^i < \mu^*} \frac{D_i \ln n}{( \mu^*-\mu^i)} + \displaystyle\sum_{i:\mu^i<\mu^*} (\mu^*-\mu^i) C_iD_i . \label{eqn10}
\end{eqnarray}
\end{proof}
%

\comment{
The following Lemma, which can be found in \cite{bremaud}, will be used in the proof of the next theorem.
\begin{lemma}
If $\left\{X_n\right\}_{n \geq 0}$ is a positive recurrent homogeneous Markov chain with state space $S$, stationary distribution $\pi$ and $\tau$ is a stopping time that is finite almost surely for which $X_{\tau}=x$ then for all $y \in S$
\begin{eqnarray*}
E\left[ \sum_{t=0}^{\tau-1} I(X_t=y) | X_0=x \right] = E[\tau | X_0=x]\pi_y
\end{eqnarray*}
\label{lastlemma}
\end{lemma}
} 

We now state the main theorem of this paper.
\begin{theorem} \label{theorem2} 
Assume all arms are finite-state, irreducible, aperiodic Markov chains whose transition probability matrices have irreducible multiplicative symmetrizations and assume all rewards are positive. 
Let $\pi^i_{\min} = \min_{x \in S^i} \pi^i_x$, $\pi_{\min} = \min_{1 \leq i \leq K} \pi^i_{\min}$, $r_{\max}=\max_{x \in S^i, 1 \leq i \leq K} r^i_x$, $S_{\max} = \max_{1 \leq i \leq K} |S^i|$, $\hat{\pi}_{\max} = \max_{x \in S^i, 1 \leq i \leq K} \left\{\pi^i_x, 1-\pi^i_x\right\}$, $\epsilon_{\min}= \min_{1\leq i \leq K} \epsilon^i$, $M^i_{\max}=\max_{x,y \in S^i, x \neq y} M^i_{x,y}$, where $\epsilon^i$ is the eigenvalue gap of the multiplicative symmetrization of the transition probability matrix of the $i$th arm and $M^i_{x,y}$ is the mean hitting time of state $y$ starting from an initial state $x$ for the $i$th arm. Then using a constant $L \geq 112 S^2_{\max} r^2_{\max} \hat{\pi}^2_{\max} /\epsilon_{\min}$, the regret of RCA can be upper bounded uniformly over time by the following, $\forall n$: 
\begin{eqnarray*}
R^{RCA}(n) &<& 4L \ln n \sum_{i: \mu^i<\mu^*} \frac{1}{\mu^*-\mu^i}\left(D_i + \frac{E_i}{\mu^*-\mu^i}\right) \\
&+& \sum_{i: \mu^i<\mu^*}C_i\left((\mu^*-\mu^i)D_i+E_i\right) +F
\end{eqnarray*}
where 
\begin{eqnarray*}
C_i &=& \left(1+\frac{(|S^i|+|S^*|)\beta} {\pi_{\min}}\right), ~~ \beta = \sum_{t=1}^{\infty} t^{-2}\\
D_i &=& \left(\frac{1}{\pi^i_{\min}} + M^i_{\max} +1\right),\\
E_i &=& \mu^i(1+M^i_{\max}) + \mu^*M^*_{\max},\\
F &=& \mu^*\left(\frac{1}{\pi_{\min}} + \max_{i \in \left\{1,...,K\right\}} M^i_{\max} +1\right).
\end{eqnarray*} 
\end{theorem}
\begin{proof}
Assume that the states which determine the regenerative sample paths are given {\em a priori} by $\gamma=[\gamma^1,\cdots,\gamma^K]$. We denote the expectations with respect to RCA given $\gamma$ as $E_\gamma$. First we rewrite the regret in the following form:
\begin{eqnarray}
R_\gamma(n)&=&\mu^* E_\gamma[T(n)] - E_\gamma[\sum_{t=1}^{T(n)} r^{\alpha(t)}_{x_{\alpha(t)}}] \nonumber \\
&&+ \mu^*E_\gamma[n-T(n)]-E_\gamma[\sum_{t=T(n)+1}^{n} r^{\alpha(t)}_{x_{\alpha(t)}}] \nonumber \\
&=& \left\{ \mu^* E_\gamma[T(n)] - \sum_{i=1}^K \mu^i E_\gamma \left[T^i(n)\right]  \right\} - Z_{\gamma}(n) \nonumber
\end{eqnarray}
\begin{eqnarray}
&& + \mu^*E_\gamma[n-T(n)]-E_\gamma[\sum_{t=T(n)+1}^{n} r^{\alpha(t)}_{x_{\alpha(t)}}] ~.  \label{eqn:last_R}
\end{eqnarray}
where for notational convenience, we have used
\begin{eqnarray}
Z_\gamma(n)=E_\gamma \left[\sum_{t=1}^{T(n)} r^{\alpha(t)}_{x_{\alpha(t)}} \right]  - \sum_{i=1}^K \mu^i E_\gamma \left[T^i(n)\right]. \nonumber
\end{eqnarray}
We can bound the first difference in (\ref{eqn:last_R}) logarithmically using Theorem \ref{theorem1}, so it remains to bound $Z_\gamma(n)$ and the last difference. We have
\begin{eqnarray}
&& Z_\gamma(n) \geq \sum_{y \in S^*}r^*_y E_\gamma\left[\sum_{j=1}^{B^*(b(n))} \sum_{X^*_t \in X^*(j)} I(X^*_t=y)\right] \nonumber \\ 
&+& \sum_{i: \mu^i < \mu^*} \sum_{y \in S^i} r^i_y E_\gamma \left[ \sum_{j=1}^{B^i(b(n))} \sum_{X^i_t \in X^i_2(j)} I(X^i_t=y)\right] \label{modif1} \\
&-& \mu^* E_\gamma \left[T^*(n)\right] \nonumber \\
&-& \sum_{i: \mu^i < \mu^*} \mu^i \left(\frac{1}{\pi^i_{\gamma^i}}+M^i_{\max}+1\right) E_\gamma \left[ B^i(b(n)) \right] ~, \nonumber
\end{eqnarray}
where the inequality comes from counting only the rewards obtained during the SB2s for all suboptimal arms. 
Applying Lemma \ref{lastlemma} to (\ref{modif1}) we get 
\begin{eqnarray}
E_\gamma \left[ \sum_{j=1}^{B^i(b(n))} \sum_{X^i_t \in X^i_2(j)} I(X^i_t=y)\right] = \frac{\pi^i_y}{\pi^i_{\gamma^i}} E_\gamma \left[B^i(b(n))\right] ~.  \nonumber
\end{eqnarray}
Rearrange terms and noting $\mu^* = \sum_{y} r^*_y \pi^*_y$, 
\begin{eqnarray}
Z_\gamma(n) \geq R^*(n) - \sum_{i: \mu^i < \mu^*} \mu^i (M^i_{\max}+1) E_\gamma \left[ B^i(b(n)) \right] \label{modif3}
\end{eqnarray}
where
\begin{eqnarray}
R^*(n)&=& \sum_{y \in S^*} r^*_y E_\gamma \left[ \sum_{j=1}^{B^*(b(n))} \sum_{X^*_t \in X^*(j)}  I(X^*_t=y)\right] \nonumber  \\
&-& \sum_{y \in S^*} r^*_y \pi^*_y  E_\gamma \left[T^*(n)\right] \nonumber.
\end{eqnarray}
Consider now $R^*(n)$. Since all suboptimal arms are played at most logarithmically, the number of time steps in which the best arm is not played is at most logarithmic.  It follows that the number of discontinuities between plays of the best arm is at most logarithmic. 
Suppose we combine successive blocks in which the best arm is played, and denote by $\bar{X}^*(j)$ the $j$-th combined block. 
Let $\bar{b}^*$ denote the total number of combined blocks up to block $b$. Each $\bar{X}^*$ thus consists of two sub-blocks: $\bar{X}^*_1$ that contains the states visited from beginning of $\bar{X}^*$ (empty if the first state is $\gamma^*$) to the state right before hitting $\gamma^*$, and sub-block $\bar{X}^*_2$ that contains the rest of $\bar{X}^*$ (a random number of regenerative cycles).  

\comment{
\begin{figure}
\includegraphics[width=3.5in]{figureb.eps}
\caption{$\bar{X}^*, \bar{X}^*_1, \bar{X}^*_2$ for a sample path}
\label{figure5}
\end{figure}
} 

Since a block $\bar{X}^*$ starts after discontinuity in playing the best arm, $\bar{b}^*(n)$ is less than or equal to total number of completed blocks in which the best arm is not played up to time $n$.  Thus 
\begin{eqnarray}
E_\gamma[\bar{b}^*(n)] \leq \sum_{i: \mu^i < \mu^*} E_\gamma[B^i(b(n))]. \label{modif4}
\end{eqnarray}
We rewrite  $R^*(n)$ in the following from:
\begin{eqnarray}
R^*(n)&=& \sum_{y \in S^*} r^*_y E_\gamma \left[ \sum_{j=1}^{\bar{b}^*(n)} \sum_{X^*_t \in \bar{X}^*_2(j)}  I(X^*_t=y)\right] \label{difference6} \\
&-& \sum_{y \in S^*} r^*_y \pi^*_y  E_\gamma \left[ \sum_{j=1}^{\bar{b}^*(n)} |\bar{X}^*_2(j)| \right] \label{difference7} 
\end{eqnarray}
\\
\begin{eqnarray}
&+& \sum_{y \in S^*} r^*_y E_\gamma \left[ \sum_{j=1}^{\bar{b}^*(n)} \sum_{X^*_t \in \bar{X}^*_1(j)}  I(X^*_t=y)\right] \label{modif5} \\
&-& \sum_{y \in S^*} r^*_y \pi^*_y  E_\gamma \left[ \sum_{j=1}^{\bar{b}^*(n)} |\bar{X}^*_1(j)| \right] \label{modif6} \\
&>& 0 - \mu^* M^*_{\max} \sum_{i: \mu^i < \mu^*} E_\gamma[B^i(b(n))] \label{difference8}
\end{eqnarray}
where the last inequality is obtained by noting the difference between (\ref{difference6}) and (\ref{difference7}) is zero by Lemma \ref{lastlemma}, using positivity of rewards to lower bound (\ref{modif5}) by $0$, and (\ref{modif4}) to upper bound (\ref{modif6}). 
Combine this with (\ref{eqn9}) and (\ref{modif3}) we can thus obtain a logarithmic upper bound on $-Z_\gamma(n)$.   
Finally, we have 
\begin{eqnarray}
&& \mu^* E_\gamma[n-T(n)]-E_\gamma[\sum_{t=T(n)+1}^{n} r^{\alpha(t)}_{x_{\alpha(t)}}] \nonumber \\
&\leq& \mu^*\left(\frac{1}{\pi_{\min}} + \max_{i \in \left\{1,...,K\right\}} M^i_{\max} +1\right) ~. \label{last_eq}
\end{eqnarray}
Therefore we have obtained the stated logarithmic bound for (\ref{eqn:last_R}).  Note that this bound does not depend on $\gamma$, and therefore is also an upper bound for $R(n)$, completing the proof. 
\end{proof}

Therefore, given minimal information about the arms such as an upper bound for $ S^2_{\max} r^2_{\max} \hat{\pi}^2_{\max} /\epsilon_{\min}$ the player can guarantee logarithmic regret by choosing an $L$ in the RCA algorithm that satisfies the condition in Theorem \ref{theorem2}. 

We end this section by noting that the logarithmic bound in $n$ is also {\em order optimal} for this restless bandit problem, i.e., no better bound than $\ln n$ is possible (however a better constant is possible). This follows from the fact that the {\em rested} bandit problem is a special case of the {\em restless} problem and in \cite{anantharam1} it is shown that the best order is logarithmic for the {\em rested} problem. Moreover, we conjecture that the order optimality of RCA holds when it is used with any index policy that is order optimal for the rested bandit problem. Because of the use of regenerative cycles in RCA, the observations used to calculate the indices can be in effect treated as coming from rested arms. Thus an approach similar to the one in Theorem \ref{theorem1} can be used to prove order optimality.

\section{An Example: Gilbert-Elliot Channel Model} \label{sec:example}


In this section we simulate RCA and two other algorithms under the commonly used Gilbert-Elliot channel model, where each channel has two states, {\em good} and {\em bad} (or $1, 0$, respectively). The first algorithm is the upper confidence bound (UCB1) algorithm from \cite{auer}.  In \cite{tekin} we have proved that it has a logarithmic regret in the case of Markovian rewards when all arms are rested, by replacing the constant $2$ in the index calculation of UCB1 with $L$ and using a result from \cite{gillman}. Using Lezaud's bound as we have done in the present paper it can be shown that this modified UCB1 algorithm, shown in Figure \ref{fig:UCB_algo}, has a logarithmic regret for $L \geq 112 S^2_{\max} r^2_{\max} \hat{\pi}^2_{\max} /\epsilon_{\min}$ for the rested bandit problem. 


\begin{figure}[htb]
\vspace{0.05in}
\fbox {
\begin{minipage}{\columnwidth }
{Upper Confidence Bound (UCB1): }
\begin{algorithmic}[1]
\STATE {Initialize: $n=1$}

\FOR {$n\leq K$}
\STATE{play arm $n$; $n:=n+1$.} 
\ENDFOR
\WHILE {$n>K$} 
\STATE{$\bar{r}^i (T^i(n)) =\frac{r^i(1)+r^i(2)+...+r^i(T^i(n))}{T^i(n)}, ~\forall i$} 
\STATE{$g^i_{n,T^i(n)}= \bar{r}^i(T^i(n)) + \sqrt{\frac{L\ln n}{T^i(n)}}, ~\forall i$} 
\STATE{play arm $j$, such that $j=\arg\max_i g^i_{n, T^i(n)}$, update $r^j(n)$ and $T^j(n)$.} 
\STATE{$n:=n+1$}
\ENDWHILE
\end{algorithmic}
\end{minipage}
} \caption{The UCB1 algorithm.} \label{fig:UCB_algo}
\end{figure}

The second algorithm is an online randomized algorithm proposed in \cite{auer2}, referred to as the Exp3 algorithm and shown in Figure \ref{fig:Exp3_algo}.  The main distinction of Exp3 is that it is a randomized algorithm: given all past observations the algorithm's current action is the outcome of a random variable.  
Randomization is helpful when rewards from arms are determined by an adversary rather than a stochastic process.  This is the context in which Exp3 is introduced and studied in \cite{auer2}.  

\comment{
\begin{framed}
\hspace{-8 mm} \textbf{Exp3} \\

{\bf Parameters}: $a \in (0,1)$\\

{\bf Initialization}: $\omega^i(1)=1, \forall i \in \left\{1,2,\cdots,K\right\}$ \\

{\bf Loop}: for any $t$, set 
\begin{eqnarray*}
p^i(t)=(1-a)\frac{\omega^i(t)}{\sum_{j=1}^K \omega^j(t)} + \frac{a}{K}
\end{eqnarray*}
Then action $\alpha(t)$ is the outcome of a random variable $X(t)$ with pmf
\begin{eqnarray*}
P(X(t)=i)=p^i(t).
\end{eqnarray*} 
If $\alpha(t)=i$ set
\begin{eqnarray*}
w^i(t+1)=w^i(t)\exp(\frac{a r^i_{x(t)}}{K p^i(t)})
\end{eqnarray*}
else $w^i(t+1)=w^i(t)$.
\end{framed}
} 

\begin{figure}[htb]
\fbox {
\begin{minipage}{\columnwidth }
{Exp3:}
\begin{algorithmic}[1]
\STATE {Initialize: select parameter $a \in (0, 1)$ and set weights $w^i(1)=1, \forall i \in \left\{1,2,\cdots,K\right\}$}
\WHILE {(1)} 
\STATE{at time $n$ compute the probabilities 
$p^i(n)=(1-a)\frac{w^i(n)}{\sum_{j=1}^K w^j(n)} + \frac{a}{K}$, $\forall i$.} 
\STATE{take a random sample of the random variable $X(n)$ with pmf: 
$P(X(n)=i)=p^i(n)$; denote the outcome by $\alpha(n)$.} 
\STATE{play arm $\alpha(n)$, and get reward $r^{\alpha(n)}$.} 
\IF{$\alpha(n)=i$} 
\STATE{set weight $w^i(n+1)=w^i(n)\exp(\frac{a r^i(n)}{K p^i(n)})$}
\ELSE
\STATE{$w^i(n+1)=w^i(n)$}
\ENDIF
\ENDWHILE
\end{algorithmic}
\end{minipage}
} \caption{The Exp3 algorithm.} \label{fig:Exp3_algo}
\end{figure}

We simulate and compare the regret of these three algorithms averaged over 100 runs, under two scenarios, denoted S1 and S2, respectively.  Each scenario involves 5 two-state channels with varying transition probabilities.  The statistics and rewards used are given in Table \ref{table1}. 
Exp3 is run with two different values of $a$: $a_1=0.1$, $a_2= \min \left\{1, \sqrt{\frac{K \ln K}{(e-1)N}}\right\}$ where $N=10^5$ is the time horizon. All arms are assumed to be in stationary distribution at the beginning. $112 S^2_{\max} r^2_{\max} \hat{\pi}^2_{\max} /\epsilon_{\min}$ is equal to $9556$ in S1 and $1037.2$ in S2. 


\comment{
\begin{center}
    \begin{tabular}{ | l | l | l |  }
    \hline
    S1  & $p_{01}$, $p_{10}$ & $r_0$, $r_1$ \\ \hline
    ch.1 & .01, .03& .1, 1 \\ \hline
    ch.2 & .04, .01& .1, 1 \\ \hline
    ch.3 & .03, .01& .1, 1 \\ \hline
    ch.4 & .02, .01& .1, 1 \\ \hline
    ch.5 & .01, .02& .1, 1 \\ \hline
    S2  & $p_{01}$, $p_{10}$ & $\pi_1$, $\mu_1$ \\ \hline
    ch.1 & .1, .2 & .1, 1  \\ \hline
    ch.2 & .1, .3 & .1, 1  \\ \hline
    ch.3 & .5, .1 & .1, 1  \\ \hline
    ch.4 & .1, .4 & .1, 1  \\ \hline
    ch.5 & .1, .5 & .1, 1  \\ \hline
    \end{tabular}
    \label{table1}
\end{center}
}

\begin{table}[h] \label{table1}
\vspace{10pt}
\begin{center}
    \begin{tabular}{ l|c|c || l |c|c}
    \hline
    S1  & $p_{01}$, $p_{10}$ & $r_0$, $r_1$ &   S2  & $p_{01}$, $p_{10}$ & $r_0$, $r_1$ \\ \hline
    ch.1 & .01, .03& .1, 1 & ch.1 & .1, .2 & .1, 1  \\ \hline
    ch.2 & .04, .01& .1, 1 & ch.2 & .1, .3 & .1, 1  \\ \hline
    ch.3 & .03, .01& .1, 1 & ch.3 & .5, .1 & .1, 1  \\ \hline
    ch.4 & .02, .01& .1, 1 & ch.4 & .1, .4 & .1, 1  \\ \hline
    ch.5 & .01, .02& .1, 1 & ch.5 & .1, .5 & .1, 1  \\ \hline
    \end{tabular}
\caption{Channel parameters} 
\vspace{-25pt}
\end{center}
\end{table}

\begin{figure}[ht]
\includegraphics[width=3.3in]{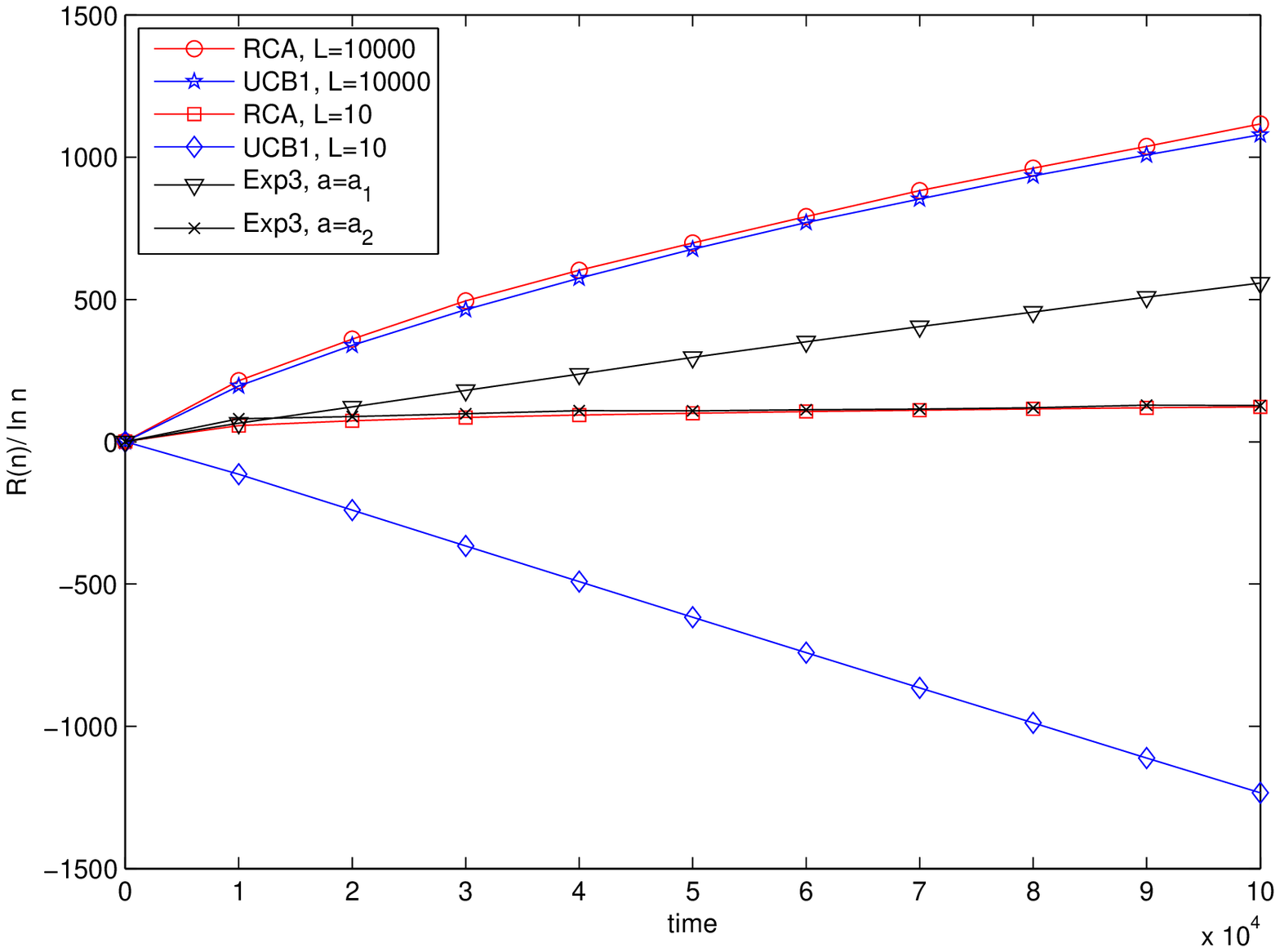}
\vspace{-15pt}
\caption{Regret under scenario S1}
\vspace{-15pt}
\label{figure1}
\end{figure}

\begin{figure}[ht]
\includegraphics[width=3.3in]{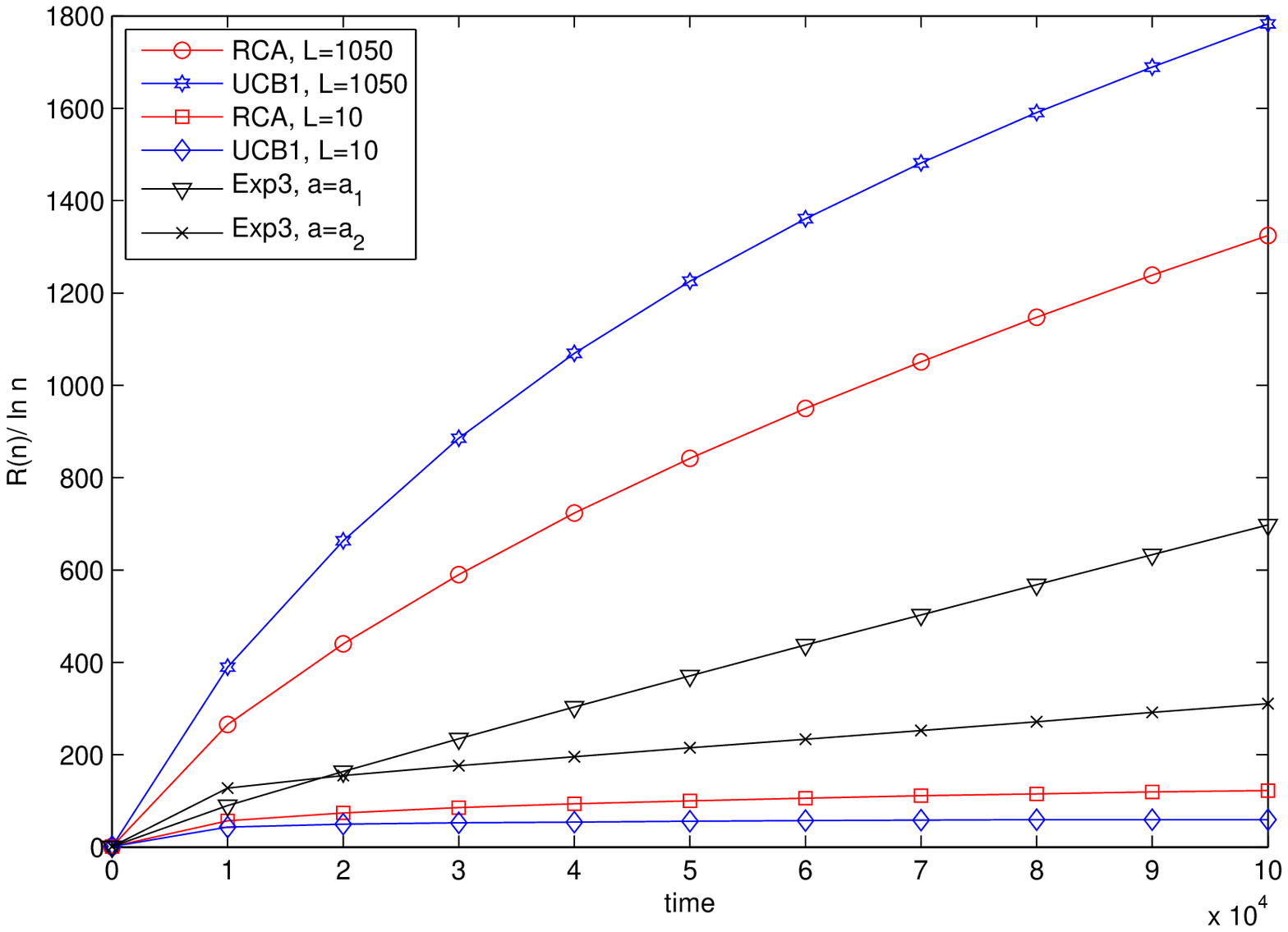}
\vspace{-15pt}
\caption{Regret under scenario S2}
\vspace{-15pt}
\label{figure2}
\end{figure}


Results are shown in Figures \ref{figure1} and \ref{figure2}, under scenarios S1 and S2, respectively.  
We make the following observations from this set of curves.  Firstly, both RCA's and UCB1's performance improves when a smaller value of $L$ is used.  This suggests that the condition $L \geq 112 S^2_{\max} r^2_{\max} \hat{\pi}^2_{\max} /\epsilon_{\min}$ is sufficient but not in general necessary for the logarithmic regret to hold.  Secondly, Exp3 shows good performance when $a_2$ is the constant choice, which utilizes the knowledge of time horizon.  If the time horizon is not given, then Exp3 has a linear regret instead as was proven in \cite{auer2}. 
Lastly, overall the performance of UCB1 is competitive compared to RCA, which has been shown to have logarithmic regret in the previous section. In particularly, in Figure \ref{figure1} for $L=10$ UCB1 outperforms RCA significantly.  This is because in this case the channels are very bursty, thus updating the indices at every time step in UCB1 is a better option than waiting for regenerative cycles to occur in RCA, which can take a long time for an update to occur.  These results suggest that there may exist logarithmic bounds for UCB1 as well.  Furthermore, they suggest obvious ways to improve the performance of RCA.  
However, as discussed earlier due to the restless nature of the arms when the indices are updated constantly the problem becomes intractable.  
It remains an interesting future study to show such bounds for UCB1.  

\vspace{5pt}
\section{Conclusion} \label{sec:conc}
\vspace{5pt}
We considered the OSA problem when the primary users' activities are modeled as generic finite-state Markov chains.  This was formulated as a single-player restless bandit problem. We proposed an algorithm that updates the sample mean based indices using regenerative sample paths and showed that its regret can be upper bounded uniformly and logarithmically over time. This is the first results showing that $\log$-regret is possible in a restless bandit learning problem.  
We numerically compare the performance of RCA with two other algorithms, UCB1 and Exp3, and conjectured that similar logarithmic bounds may exist for UCB1 as well. 
\vspace{5pt}
\bibliographystyle{IEEE}
\bibliography{OSA}

\begin{thebibliography}{10}
\providecommand{\url}[1]{#1}
\csname url@samestyle\endcsname
\providecommand{\newblock}{\relax}
\providecommand{\bibinfo}[2]{#2}
\providecommand{\BIBentrySTDinterwordspacing}{\spaceskip=0pt\relax}
\providecommand{\BIBentryALTinterwordstretchfactor}{4}
\providecommand{\BIBentryALTinterwordspacing}{\spaceskip=\fontdimen2\font plus
\BIBentryALTinterwordstretchfactor\fontdimen3\font minus
  \fontdimen4\font\relax}
\providecommand{\BIBforeignlanguage}[2]{{%
\expandafter\ifx\csname l@#1\endcsname\relax
\typeout{** WARNING: IEEEtran.bst: No hyphenation pattern has been}%
\typeout{** loaded for the language `#1'. Using the pattern for}%
\typeout{** the default language instead.}%
\else
\language=\csname l@#1\endcsname
\fi
#2}}
\providecommand{\BIBdecl}{\relax}
\BIBdecl

\bibitem{auer2}
P.~Auer, N.~Cesa-Bianchi, Y.~Freund, and R.~Schapire, ``The nonstochastic
  multiarmed bandit problem,'' \emph{SIAM Journal on Computing}, vol.~32, pp.
  48--77, 2002.

\bibitem{robbins2}
H.~Robbins, ``Some aspects of the sequential design of experiments,''
  \emph{Bull. Amer. Math. Soc.}, vol.~55, pp. 527--535, 1952.

\bibitem{lai1}
T.~Lai and H.~Robbins, ``Asymptotically efficient adaptive allocation rules,''
  \emph{Advances in Applied Mathematics}, vol.~6, pp. 4--22, 1985.

\bibitem{anantharam2}
V.~Anantharam, P.~Varaiya, and J.~. Walrand, ``Asymptotically efficient
  allocation rules for the multiarmed bandit problem with multiple plays-part
  i: Iid rewards,'' \emph{IEEE Trans. Automat. Contr.}, pp. 968--975, November
  1987.

\bibitem{anantharam1}
------, ``Asymptotically efficient allocation rules for the multiarmed bandit
  problem with multiple plays-part ii: Markovian rewards,'' \emph{IEEE Trans.
  Automat. Contr.}, pp. 977--982, November 1987.

\bibitem{agrawal1}
R.~Agrawal, ``Sample mean based index policies with o(log n) regret for the
  multi-armed bandit problem,'' \emph{Advances in Applied Probability},
  vol.~27, no.~4, pp. 1054--1078, December 1995.

\bibitem{auer}
P.~Auer, N.~Cesa-Bianchi, and P.~Fischer, ``Finite-time analysis of the
  multiarmed bandit problem,'' \emph{Machine Learning}, vol.~47, p. 235–256,
  2002.

\bibitem{tekin}
C.~Tekin and M.~Liu, ``Online algortihms for the multi-armed bandit problem
  with markovian rewards, http://arxiv.org/abs/1007.2238.''

\bibitem{liu2}
K.~Liu and Q.~Zhao, ``Distributed learning in multi-armed bandit with multiple
  players, http://arxiv.org/abs/0910.2065.''

\bibitem{anandkumar}
A.~Anandkumar, N.~Michael, and A.~Tang, ``Opportunistic spectrum access with
  multiple players: Learning under competition,'' in \emph{Proc. of IEEE
  INFOCOM}, March 2010.

\bibitem{krishnamachari}
Y.~Gai, B.~Krishnamachari, and R.~Jain, ``Learning multiuser channel
  allocations in cognitive radio networks: a combinatorial multi-armed bandit
  formulation,'' in \emph{IEEE Symp. on Dynamic Spectrum Access Networks
  (DySPAN)}, April 2010.

\bibitem{gittins1}
J.~Gittins, ``Bandit processes and dynamic allocation indices,'' \emph{Journal
  of the Royal Statistical Society}, vol.~41, no.~2, pp. 148--177, 1979.

\bibitem{whittle}
P.~Whitlle, ``Restless bandits,'' \emph{J. Appl. Prob.}, pp. 301--313, 1988.

\bibitem{ahmad}
S.~H.~A. Ahmad, M.~Liu, T.~Javidi, Q.~Zhao, and B.~Krishnamachari, ``Optimality
  of myopic sensing in multi-channel opportunistic access,'' \emph{IEEE
  Transactions on Information Theory}, vol.~55, no.~9, pp. 4040--4050,
  September 2009.

\bibitem{gillman}
D.~Gillman, ``A chernoff bound for random walks on expander graphs,''
  \emph{SIAM J. Comp.}, vol.~27, no.~4, p. 1203–1220, 1998.

\bibitem{lezaud}
Q.~Zhao and B.~Sadler, ``Chernoff-type bound for finite markov chains,''
  \emph{Ann. Appl. Prob.}, vol.~8, pp. 849--867, 1998.

\bibitem{bremaud}
P.~Bremaud, \emph{Markov Chains, Gibbs Fields, Monte Carlo Simulation and
  Queues}.\hskip 1em plus 0.5em minus 0.4em\relax Springer, 1998.

\end{thebibliography}

\end{document}